\documentclass[12pt]{amsart}
\usepackage{amssymb}
\usepackage[all]{xy}

\usepackage{verbatim}
\usepackage{eucal}
\usepackage{latexsym}
\usepackage{amsfonts,amssymb,amsxtra}
\usepackage{relsize}
\usepackage[dvips]{color}
\usepackage{latexsym}
\usepackage{mathptmx}
\usepackage{amssymb}
\usepackage{graphicx}
\usepackage{amsthm}
\usepackage{amsfonts}
\usepackage{amscd}
\usepackage{bm}
\usepackage[utf8]{inputenc}

\usepackage{amsmath}
\usepackage{amssymb}

\usepackage{verbatim} 
\usepackage{eucal} 
\usepackage{color}

\usepackage[draft]{todonotes}
\usepackage{pgf}
\usepackage{tikz}
\usetikzlibrary{arrows,automata}

\usepackage{pgf, tikz}
\definecolor {processblue}{cmyk}{0.96,0,0,0}
\input xy
\xyoption{all}
\usepackage{mathtools}
\usepackage{amsfonts}
\usepackage{amsthm}

\usepackage{amscd}
  \newtheorem{The}{Theorem}[section]
  \newtheorem{Pro}[The]{Proposition}
  
  \newtheorem{Cor}[The]{Corollary}
  \newtheorem{Def}[The]{Definition}
  \newtheorem{Rem}[The]{Remark}

\newcommand{\bsm}{\begin{smallmatrix}}
\newcommand{\esm}{\end{smallmatrix}}
\newcommand{\bbm}{\begin{matrix}}
\newcommand{\ebm}{\end{matrix}}

\usepackage{eucal}
\usepackage{mathptmx}
\usepackage{mathrsfs}

\newcommand{\Hom}{\rm{Hom}}

\theoremstyle{definition}

\theoremstyle{plain}

\theoremstyle{definition}

\numberwithin{equation}{section}

\begin{document}

\title{The rigidity of filtered colimits of $n$-cluster tilting subcategories}
\newcommand\shortTitle{Finiteness for n-cluster tilting subcategories of the module category}
\author{Ziba Fazelpour}
\address{Department of Pure Mathematics\\
Faculty of Mathematics and Statistics\\
University of Isfahan\\
P.O. Box: 81746-73441, Isfahan, Iran}
\email{z.fazelpour@ipm.ir}
\author{Alireza Nasr-Isfahani}
\address{Department of Pure Mathematics\\
Faculty of Mathematics and Statistics\\
University of Isfahan\\
P.O. Box: 81746-73441, Isfahan, Iran\\ and School of Mathematics, Institute for Research in Fundamental Sciences (IPM), P.O. Box: 19395-5746, Tehran, Iran}
\email{nasr$_{-}$a@sci.ui.ac.ir / nasr@ipm.ir}



\maketitle

\begin{abstract}
Let $\Lambda$ be an artin algebra and $\mathcal{M}$ be an $n$-cluster tilting subcategory of $\Lambda$-mod with $n \geq 2$.  From the viewpoint of higher homological algebra, a question that naturally arose in \cite{en1} is when $\mathcal{M}$ induces an $n$-cluster tilting subcategory of $\Lambda$-Mod. In this paper, we answer this question and explore its connection to Iyama’s question on the finiteness of $n$-cluster tilting subcategories of $\Lambda$-mod. In fact, our theorem reformulates Iyama’s question in terms of the vanishing of Ext; and highlights its relation with the rigidity of filtered colimits of $\mathcal{M}$. Also, we show that ${\rm Add}(\mathcal{M})$ is an $n$-cluster tilting subcategory of $\Lambda$-Mod if and only if ${\rm Add}(\mathcal{M})$ is a maximal $n$-rigid subcategory of $\Lambda$-Mod if and only if $\lbrace X\in  \Lambda$-{\rm Mod}$~|~ {\rm Ext}^i_{\Lambda}(\mathcal{M},X)=0 ~~~ {\rm for ~all}~ 0<i<n \rbrace \subseteq {\rm Add}(\mathcal{M})$ if and only if $\mathcal{M}$ is of finite type if and only if ${\rm Ext}_{\Lambda}^1({\underrightarrow{\lim}}\mathcal{M}, {\underrightarrow{\lim}}\mathcal{M})=0$. Moreover, we present several equivalent conditions for the Iyama’s question which shows the relation of the Iyama's question with different subjects in representation theory such as purity and covering theory. 
\end{abstract}

\section{Introduction}

One of the central goals in the representation theory of finite dimensional algebras is to classify finitely generated indecomposable  modules up to isomorphism and morphisms between them. But this is considered as "impossible" in general because most algebras have wild representation type. In the modern representation theory of algebras, the main approach to understanding the structure of the module category is to investigate subcategories of the module category having good properties. 

\vspace{2mm}
In higher-dimensional Auslander-Reiten theory, as introduced by Osamu Iyama in \cite{I3, I2}, the module category is replaced by a subcategory with suitable homological properties known as an $n$-cluster tilting subcategory, where $n$ is a positive integer. Note that $\Lambda$-mod itself is the unique 1-cluster tilting subcategory of $\Lambda$-mod. If an $n$-cluster tilting subcategory $\mathcal{M}$ admits an additive generator, then $\mathcal{M}$ is called of finite type. The question of finiteness of $n$-cluster tilting subcategories for $n \geq 2$, which is among the first that have been asked by Iyama \cite{I3}, is still open. Up to now, all known $n$-cluster tilting subcategories with $n \geq 2$ are of finite type. Several equivalence conditions for finiteness of $n$-cluster tilting subcategories are given in \cite{dn, dn1, en, en1, fin}.

Let $\mathcal{M}$ be an $n$-cluster tilting subcategory of $\Lambda$-mod with $n \geq 2$. It is natural to ask the following question: when does $\mathcal{M}$ induce an $n$-cluster tilting subcategory in $\Lambda$-Mod? Ebrahimi and the second author in \cite{en, en1} investigated this question and its relationship to the Iyama's question. They showed that ${\rm Add}(\mathcal{M})$ is an $n$-cluster tilting subcategory of $\Lambda$-Mod if and only if $\mathcal{M}$ is of finite type (see \cite[Theorem 4.12]{en}). For an $n$-cluster tilting subcategory $\mathcal{M}$ of $\Lambda$-mod,  ${\underrightarrow{\lim}}\mathcal{M}$, consists of all filtered colimits of objects from $\mathcal{M}$, is a generating-cogenerating functorially finite subcategory of $\Lambda$-Mod. Hence the following question, as has been posited in \cite[Question 3.5]{en1}, naturally arises.

\vspace{1mm}
\textbf{Question A:} Let $\mathcal{M}$ be an $n$-cluster tilting subcategory of $\Lambda$-mod with $n \geq 2$. Is ${\underrightarrow{\lim}}\mathcal{M}$ an $n$-rigid subcategory of $\Lambda$-Mod?

\vspace{1mm}
They proved ${\underrightarrow{\lim}}\mathcal{M}$ is the unique $n$-cluster tilting subcategory of $\Lambda$-Mod containing $\mathcal{M}$ if Question A has positive answer (see \cite[Proposition 4.1]{en1}). In the case $n=2$, employing cotorsion theory, they showed Question A is equivalent to the Iyama’s Question (see \cite[Theorem 4.10]{en1}).  

\vspace{2mm}
In this paper, we study these questions and their connection to the Iyama's question. First, we focus on the subcategory ${\rm Add}(\mathcal{M})$ of $\Lambda$-Mod and we prove the following theorem.

\vspace{1mm}
\textbf{Theorem A.} (Theorem \ref{3.1}) Let $\mathcal{M}$ be an $n$-cluster tilting subcategory of $\Lambda$-mod with $n \geq 2$. Then the following statements are equivalent.\vspace{-1mm}

\begin{itemize}
\item[$(a)$] $\mathcal{M}$ is of finite type.
\item[$(b)$] $\lbrace X\in  \Lambda$-{\rm Mod}$~|~ {\rm Ext}^i_{\Lambda}(\mathcal{M},X)=0 ~~~ {\rm for ~all}~ 0<i<n \rbrace = {\rm Add}(\mathcal{M})$.
\item[$(c)$] $\lbrace X\in  \Lambda$-{\rm Mod}$~|~ {\rm Ext}^i_{\Lambda}(\mathcal{M},X)=0 ~~~ {\rm for ~all}~ 0<i<n \rbrace \subseteq {\rm Add}(\mathcal{M})$.
\item[$(d)$] ${\rm Add}(\mathcal{M})$ is an $n$-cluster tilting subcategory of $\Lambda$-{\rm Mod}.
\end{itemize}

Next we show that Question A is equivalent to the Iyama's question when $n \geq 2$. In fact, we prove the following theorem.\\

\textbf{Theorem B.} (Theorem \ref{2.2}) Let $\mathcal{M}$ be an $n$-cluster tilting subcategory of $\Lambda$-mod with $n \geq 2$. Then the following statements are equivalent.\vspace{-1mm}
\begin{itemize}
\item[$(a)$] $\mathcal{M}$ is of finite type.
\item[$(b)$] ${\underrightarrow{\lim}}\mathcal{M}$ is an $n$-cluster tilting subcategory of $\Lambda$-{\rm Mod}.
\item[$(c)$] ${\underrightarrow{\lim}}\mathcal{M}$ is an $n$-rigid subcategory of $\Lambda$-{\rm Mod}.
\item[$(d)$] ${\rm Ext}_{\Lambda}^1({\underrightarrow{\lim}}\mathcal{M}, {\underrightarrow{\lim}}\mathcal{M})=0$.
\item[$(e)$] Every pure submodule $N$ of a module $X$ in ${\rm Add}(\mathcal{M})$ is a direct summand of $X$.
\item[$(f)$] ${\underrightarrow{\lim}}\mathcal{M}$ is pure semisimple.
\end{itemize}

Mittag-Leffler modules occur naturally in algebra, algebraic geometry, and model theory (see \cite{gro, rg, mp}). Mittag-Leffler modules play a prominent role in the Simson’s characterization of perfect functor categories (see \cite[Theorem 6.3]{s77} and see also \cite[Theorem 8]{af}). We study the subcategory of all  Mittag-Leffler modules in ${\underrightarrow{\lim}} \mathcal{M}$ when $\mathcal{M}$ is an $n$-cluster tilting subcategory of $\Lambda$-mod with $n \geq 2$. As a consequence we show the assumption that every module in ${\underrightarrow{\lim}} \mathcal{M}$ is Mittag-Leffler guarantees that $\mathcal{M}$ is of finite type (see Corollary \ref{2.3}).

\subsection{Notation }
Throughout this paper all rings are associative with unit unless otherwise stated and $n$ always denotes a fixed positive integer. Also we assume that $\Lambda$ is an artin algebra. Moreover, we write homomorphisms on the right. Let $T$ be a ring, not necessary with unit. We denote by $T$-Mod (resp., $T$-mod) the category of all  (finitely generated) left $T$-modules. We write $\prod_{i \in I}U_i$ (resp., $\bigoplus_{i \in I}U_i$) for the direct product (resp., direct coproduct) of the left $T$-modules $U_i$ in $T$-Mod. A left (resp., right) $T$-module $M$ is called {\it unitary} if $TM=M$ (resp., $MT=M$). We denote by $T$Mod (resp., Mod$T$) the category of all unitary left (resp., right) $T$-modules. Also, we denote by ${\rm Proj}(T)$ (resp., ${\rm proj}(T)$) the full subcategory of $T$Mod consisting of (resp., finitely generated) unitary left $T$-modules which are projective in $T$Mod.  Let $\mathcal{M}$ be an additive full subcategory of $\Lambda$-mod. We denote by ${\rm Add}(\mathcal{M})$ (resp., ${\rm add}(\mathcal{M})$) the full subcategory of $\Lambda$-Mod (resp., $\Lambda$-mod) whose objects are direct summands of (resp., finite) direct sums of objects in $\mathcal{M}$. Also, we denote by ${\underrightarrow{\lim}} \mathcal{M}$ the full subcategory of $\Lambda$-Mod in which each of its objects is a direct limit of objects in $\mathcal{M}$.

\section{Preliminaries} 

Throughout this section, we fix a positive integer $n$. First, we recall the definition of $n$-cluster tilting subcategories of abelian categories. Then, we investigate the weak global dimension of (Gabriel) functor rings of n-cluster tilting subcategories of module categories, generalizing results from Fuller \cite[Proposition 1.5]{f} and Auslander \cite[Proposition 4.2]{Ausla2} (see Theorem \ref{1.2}). By using this outcome, we prove some basic properties of the subcategory of filtered colimit ${\underrightarrow{\lim}}\mathcal{M}$ of $n$-cluster tilting subcategory $\mathcal{M}$ of $\Lambda$-mod which are crucial in our investigation.

Before presenting the definition of the $n$-cluster tilting subcategories of arbitrary abelian categories, let us recall some notions.\\

Let $\mathcal{M}$ be a full subcategory of an abelian category $\mathcal{A}$. We recall that $\mathcal{M}$ is {\it generating} if any object in $\mathcal{A}$ is a quotient of an object in $\mathcal{M}$; that is, for every object $A \in \mathcal{A}$, there exists an exact sequence $M \rightarrow A \rightarrow 0$ with $M \in \mathcal{M}$. {\it Cogenerating subcategories} are defined dually. $\mathcal{M}$ is called a {\it generating-cogenerating subcategory} of $\mathcal{A}$ if it is both generating and cogenerating. 

Given an object $A \in \mathcal{A}$, a morphism $\theta: M \rightarrow A$ with $M \in \mathcal{M}$ is called a {\it right $\mathcal{M}$-approximation of $A$} if any morphism from an object in $\mathcal{M}$ to $A$ factors through $\theta$. The subcategory $\mathcal{M}$ is called {\it contravariantly finite} if every object in $\mathcal{A}$ admits a right $\mathcal{M}$-approximation. Left $\mathcal{M}$-approximations and covariantly finite subcategories are defined dually. Also $\mathcal{M}$ is called {\it functorially finite} if it is both contravariantly and covariantly finite. We refer the reader to \cite{ar, as} for more details on functorially finite subcategories.

\begin{Def}{\rm (See }\cite[Definition 3.14]{J}{\rm )}
{\rm Let $\mathcal{A}$ be an abelian category and $\mathcal{M}$ be a generating-cogenerating full subcategory of $\mathcal{A}$. The subcategory $\mathcal{M}$ of $\mathcal{A}$ is called {\it $n$-cluster tilting} if it is functorially finite in $\mathcal{A}$ and $\mathcal{M}^{\bot_n}=\mathcal{M}={^{\bot_n}\mathcal{M}}$, where
\begin{center}
${^{\bot_n}\mathcal{M}}=\lbrace X\in \mathcal{A}~|~ {\rm Ext}^i_{\mathcal{A}}(X,\mathcal{M})=0 ~~~ {\rm for ~all }~ 0<i<n \rbrace$\\
$\mathcal{M}^{\bot_n}=\lbrace Y\in  \mathcal{A}~|~ {\rm Ext}^i_{\mathcal{A}}(\mathcal{M},Y)=0 ~~~ {\rm for ~all}~ 0<i<n \rbrace$.
\end{center}
Note that $\mathcal{A}$ itself is the unique 1-cluster tilting subcategory of $\mathcal{A}$.} 
\end{Def}

\noindent
Let $\mathcal{M}$ be an $n$-cluster tilting subcategory of $\Lambda$-mod. Assume that $\lbrace U_{\alpha}~|~ \alpha \in J \rbrace$ is a complete set of non-isomorphic indecomposable modules in $\mathcal{M}$. Set $U=\bigoplus_{\alpha \in J} U_{\alpha}$ and for each $\alpha \in J$ let $e_{\alpha}=\pi_{\alpha} \varepsilon_{\alpha}$, where $\pi_{\alpha}:U \rightarrow U_{\alpha}$ is the canonical projection and $\varepsilon_{\alpha}:U_{\alpha} \rightarrow U$ is the canonical injection. For each object $X$ of $\Lambda$-Mod, we define as in \cite[p. 242]{men},  $$\widehat{{\Hom}}_{\Lambda}(U,X)=\lbrace f \in {\rm Hom}_{\Lambda}(U,X)~|~ e_{\alpha}f =0~ {\rm for ~almost~ all}~ \alpha \in J \rbrace.$$  For $X=U$, we write $T:= \widehat{{\rm Hom}}_{\Lambda}(U, U)=\widehat{{\rm End}}_{\Lambda}(U)$. Following \cite[p. 40]{fu} (see also \cite[p. 370]{y}), $T=\widehat{{\rm End}}_{\Lambda}(U)$ is called {\it functor ring {\rm(or} Gabriel ring{\rm )} of $\mathcal{M}$}. 

A ring $S$ (not necessary with unit) has {\it enough idempotents} if there exists a family $\lbrace q_{\alpha}~|~\alpha \in I \rbrace$ of pairwise orthogonal idempotents of $S$ such that $S=\bigoplus_{\alpha \in I}Sq_{\alpha}=\bigoplus_{\alpha \in I}q_{\alpha}S$ (see \cite[p. 39]{fu}). From \cite[Proposition 2.2(6)]{men}, we can see that $T=\bigoplus_{\alpha \in J}Te_{\alpha}=\bigoplus_{\alpha \in J}e_{\alpha}T$ is a ring with enough idempotents. Moreover, the assignment $X \mapsto \widehat{{\Hom}}_{\Lambda}(U,X)$ defines a covariant left exact functor $\widehat{{\Hom}}_{\Lambda}(U,-): \Lambda$-Mod$ \rightarrow T$Mod (see \cite[Proposition 2.2(5)]{men}). The functor $\widehat{{\Hom}}_{\Lambda}(U,-): \Lambda$-Mod $\rightarrow T$Mod preserves coproducts and induces an additive equivalence between the full subcategory ${\rm Add}(\mathcal{M})$ of $\Lambda$-Mod and the full subcategory  ${\rm Proj}(T)$ of $T$Mod with the inverse equivalence $U \otimes_T -$ (see also \cite[p. 40-41]{fu}). Hence, we can see that $\widehat{{\Hom}}_{\Lambda}(U,-): \mathcal{M} \rightarrow {\rm proj}(T)$ is an additive equivalence with the inverse equivalence $U \otimes_{T}-$. By \cite[Theorem 3.6]{zn}, the functor ring of $\mathcal{M}$ uniquely determines up to Morita equivalence (see also the proof of \cite[Theorem 3.3]{gs}).\\

Let $S$ be a ring with enough idempotents. As usual, the projective dimension of a unitary left $S$-module $X$ which is denoted by ${\rm pd}(X)$, is the infimum of the integers $d$ such that there exists an exact sequence $$0 \rightarrow P_d \rightarrow P_{d-1} \rightarrow  \cdots \rightarrow P_0 \rightarrow X \rightarrow 0$$ with each $P_i$ is a projective unitary left $S$-module (see \cite[p. 208]{ai}).

\begin{Pro}\label{1.2.2}
Let $\mathcal{M}$ be an $n$-cluster tilting subcategory of $\Lambda$-mod and $T$ be the functor ring of $\mathcal{M}$. Then the projective dimension of the kernel of any morphism in ${\rm proj}(T)$ is less than or equal to $n-1$.
\end{Pro}
\begin{proof}
Let $\lbrace U_{\alpha}~|~ \alpha \in J \rbrace$ be a complete set of non-isomorphic indecomposable modules in $\mathcal{M}$ and $T=\widehat{{\rm End}}_{\Lambda}(U)$, where $U=\bigoplus_{\alpha \in J} U_{\alpha}$. Let $h: Q' \rightarrow Q$ be a morphism in ${\rm proj}(T)$. We show that ${\rm pd}({\rm Ker}(h)) \leq n-1$. Since the functor $\widehat{{\Hom}}_{\Lambda}(U,-): \Lambda$-Mod$ \rightarrow T$Mod is left exact and induces an additive equivalence between the full subcategory $\mathcal{M}$ of $\Lambda$-Mod and the full subcategory  ${\rm proj}(T)$ of $T$Mod, we have the following commutative diagram 
\begin{displaymath}
\xymatrix{
0 \ar[r] & {\rm Ker}(h) \ar[r]  & Q' \ar[d]^{\cong} \ar[r]^{h} & Q \ar[d]^{\cong} \\
0 \ar [r] & \widehat{{\Hom}}_{\Lambda}(U, {\rm Ker}(f)) \ar[r] &  \widehat{{\Hom}}_{\Lambda}(U,M')  \ar[r]^{\overline{f}} & \widehat{{\Hom}}_{\Lambda}(U,M)}
\end{displaymath}
where $f: M' \rightarrow M$ is a morphism in $\mathcal{M}$ and $\overline{f}=\widehat{{\Hom}}_{\Lambda}(U, f)$. Hence ${\rm Ker}(h) \cong \widehat{{\Hom}}_{\Lambda}(U, {\rm Ker}(f))$ as $T$-modules. Therefore it is enough to show that ${\rm pd}(\widehat{{\Hom}}_{\Lambda}(U, {\rm Ker}(f))) \leq n-1$.  If $n=1$, then $\mathcal{M}=\Lambda$-mod. Since ${\rm Ker}(f)$ is a finitely generated left $\Lambda$-module, $\widehat{{\Hom}}_{\Lambda}(U, {\rm Ker}(f)) \in {\rm proj}(T)$ and so the result follows. Now let $n \geq 2$. Since ${\rm Ker}(f)$ is finitely generated, by \cite[Proposition 2.3]{I3}, we have an exact sequence $$0 \rightarrow M_{n-1} \overset{\beta_{n-1}}{\rightarrow} \cdots \rightarrow M_0 \overset{\beta_{0}}{\rightarrow} {\rm Ker}(f) \rightarrow 0$$ with terms in $\mathcal{M}$ such that the sequence
$$0 \rightarrow {\rm Hom}_{\mathcal{M}}(-,M_{n-1}) \overset{(-,\beta_{n-1})}{\rightarrow} \cdots \rightarrow {\rm Hom}_{\mathcal{M}}(-,M_0) \overset{(-,\beta_{0})}{\rightarrow} {\rm Hom}_{\mathcal{M}}(-,{\rm Ker}(f)) \rightarrow 0  \hspace{1cm} (\star)$$
is exact. Now we show that the following sequence is exact 
$$0 \rightarrow \widehat{{\Hom}}_{\Lambda}(U,M_{n-1}) \overset{\overline{\beta_{n-1}}}{\rightarrow} \cdots \rightarrow \widehat{{\Hom}}_{\Lambda}(U,M_0) \overset{\overline{\beta_{0}}}{\rightarrow} \widehat{{\Hom}}_{\Lambda}(U,{\rm Ker}(f)) \rightarrow 0, \hspace{2cm} (\star \star)$$ 
where $\overline{\beta_{i}}=\widehat{{\Hom}}_{\Lambda}(U,\beta_{i})$ for each $i \in \lbrace 0, \cdots, n-1 \rbrace$. Since the functor $\widehat{{\Hom}}_{\Lambda}(U,-): \Lambda$-Mod$ \rightarrow T$Mod is left exact, the sequence $(\star \star)$ is exact at $\widehat{{\Hom}}_{\Lambda}(U,M_{n-1})$ and $\widehat{{\Hom}}_{\Lambda}(U,M_{n-2})$. Let $l \in \lbrace -1, 0, \cdots, n-3\rbrace$. Set   $M_{-1}= {\rm Ker}(f)$ and $\beta_{-1}=0$. It is sufficient to show that it is exact at $\widehat{{\Hom}}_{\Lambda}(U,M_l)$. Let $g \in \widehat{{\Hom}}_{\Lambda}(U,M_l)$ such that $g \beta_{l} = 0$. For each $\alpha \in J$ let $e_{\alpha}=\pi_{\alpha} \varepsilon_{\alpha}$, where $\pi_{\alpha}:U \rightarrow U_{\alpha}$ is the canonical projection and $\varepsilon_{\alpha}:U_{\alpha} \rightarrow U$ is the canonical injection. Since $g \in \widehat{{\Hom}}_{\Lambda}(U,M_l)$, there exists a finite subset $A$ of $J$ such that  $e_{\alpha}g = 0$ for all $\alpha \in J \setminus A$. Consider the following diagram
\begin{displaymath}
\xymatrix{
& \bigoplus_{\alpha \in A}U_{\alpha}  \ar@/^/[d]^{\varepsilon} &\\
& U  \ar@/^/@<-1ex>[u]^{\pi} \ar[d]_{g}   & \\
M_{l+1}  \ar[r]^{\beta_{l+1}} & M_l  \ar[r]^{\beta_{l}} & M_{l-1}}
\end{displaymath}
where $M_{-2}=0$ and $\varepsilon: \bigoplus_{\alpha \in A}U_{\alpha} \rightarrow U$ and $\pi: U \rightarrow \bigoplus_{\alpha \in A}U_{\alpha}$ are the canonical injection and projection, respectively. Since $\bigoplus_{\alpha \in A}U_{\alpha} \in \mathcal{M}$, $\varepsilon g \beta_{l} = 0$, and the sequence $(\star)$ is exact, there exists a morphism $\gamma: \bigoplus_{\alpha \in A}U_{\alpha} \rightarrow M_{l+1}$ such that $\gamma \beta_{l+1} = \varepsilon g$ and so $\pi \gamma \beta_{l+1} = \pi \varepsilon g$. It is not difficult to see that $\pi \varepsilon g =g$ and  $\pi \gamma \in \widehat{{\Hom}}_{\Lambda}(U,M_{l+1})$. This implies that the sequence $(\star \star)$ is exact and so the result follows.
\end{proof}

Let $R$ be a ring with enough idempotents. A unitary left $R$-module $N$ is called {\it flat} if the functor $- \otimes_{R}N: {\rm Mod}R \rightarrow \mathfrak{Ab}$ is an exact functor, where $\mathfrak{Ab}$ is the category of abelian groups (see \cite[p. 115]{gs}). Note that a unitary left $R$-module $N$ is flat if and only if $N$ is a direct limit of finitely generated projective unitary left $R$-modules (see \cite[Proposition 49.5]{wi}). Moreover, we denote by ${\rm Flat}(T)$ the full subcategory of $T{\rm Mod}$ consisting of flat unitary left $T$-modules.  

Given a unitary left $R$-module $M$, the flat dimension of $M$ which is denoted by ${\rm fd}(M)$, is less than or equal to $m$ if there exists a finite flat resolution $$0 \rightarrow F_m \rightarrow \cdots \rightarrow F_1 \rightarrow F_0 \rightarrow M \rightarrow 0.$$
If no such finite resolution exists, then ${\rm fd}(M)=\infty$. The weak global dimension of rings with enough idempotents is defined as usual \cite[p. 98]{f}. The {\it weak global dimension} of $R$ which is denoted by ${\rm w. gl. dim}(R)$ is defined as a supremum of the flat dimension of all unitary left $R$-modules (see \cite[p. 210]{de}).

\begin{The}\label{1.2}
Let $\mathcal{M}$ be an $n$-cluster tilting subcategory of $\Lambda$-mod and $T$ be the functor ring of $\mathcal{M}$. Then ${\rm w.gl.dim}(T) \leq n+1$.
\end{The}
\begin{proof}
For any $N \in T$Mod, we take an exact sequence $$0 \rightarrow K \rightarrow Q_n \rightarrow \cdots \rightarrow Q_1 \overset{g}{\rightarrow} Q_0 \rightarrow N \rightarrow 0$$ in $T$Mod, where each $Q_i$ is projective in $T$Mod. By applying the technique used in the proof of \cite[Theorem 3.1(1)]{fin}, we can see that ${\rm Ker}(g)\cong {\underrightarrow{\lim}} {\rm Ker}(g_i)$ as $T$-modules, where each $g_i$ is a morphism in ${\rm proj}(T)$. By Proposition \ref{1.2.2}, ${\rm pd}({\rm Ker}(g_i)) \leq n-1$. Let $X$ be a unitary right $T$-module. Then ${\rm Tor}_n^{T}(X, {\rm Ker}(g_i))=0$. Since  ${\rm Tor}_n^{T}(X, {\rm Ker}(g)) \cong {\underrightarrow{\lim}} {\rm Tor}_n^{T}(X, {\rm Ker}(g_i))$, we have ${\rm Tor}_n^{T}(X, {\rm Ker}(g)) = 0$. This yields $K$ is a flat unitary left $T$-module. Therefore ${\rm fd}(N) \leq n+1$ and so ${\rm w.gl.dim}(T) \leq n+1$.
\end{proof}

As a consequence, we have the following results.

\begin{Cor} {\rm (}\cite[Proposition 1.5]{f}{\rm )}
Let $T$ be the functor ring of $\Lambda$-mod. Then ${{\rm w.gl.dim}(T) \leq 2}$.
\end{Cor}

$\Lambda$ is called {\it of finite type} if $\Lambda$-mod has an additive generator.

\begin{Cor}{\rm (}\cite[Proposition 4.2]{Ausla2}{\rm )}
Let $\Lambda$ be of finite type. Then ${\rm gl.dim}({\rm End}_{\Lambda}(M)) \leq 2$, where $M$ is an additive generator of $\Lambda$-mod.
\end{Cor}

\section{Main results}

In this section we prove the main results of the paper. A full subcategory $\mathcal{M}$ of an abelian category $\mathcal{A}$ is called {\it $n$-rigid}, if ${\rm Ext}_{\mathcal{A}}^{k}(\mathcal{M}, \mathcal{M})=0$ for every $k \in \lbrace 1, \cdots, n-1\rbrace$ (see \cite[p. 443]{b}). Clearly any $n$-cluster tilting subcategory $\mathcal{M}$ of $\mathcal{A}$ is $n$-rigid.\\

In the following proposition, we list some basic properties of the subcategory ${\underrightarrow{\lim}}\mathcal{M}$ of $\Lambda$-Mod when $\mathcal{M}$ is an $n$-cluster tilting subcategory of $\Lambda$-mod.

\begin{Pro}\label{2.1}
Let $\mathcal{M}$ be an $n$-cluster tilting subcategory of $\Lambda$-mod with $n \geq 2$. Then
\begin{itemize}
\item[$(a)$] ${\underrightarrow{\lim}}\mathcal{M}$ is a generating-cogenerating subcategory of $\Lambda$-{\rm Mod}.
\item[$(b)$] ${\underrightarrow{\lim}}\mathcal{M}$ is a functorially finite subcategory of $\Lambda$-{\rm Mod}.
\item[$(c)$] ${\rm Ext}_{\Lambda}^{k}({\rm Add}(\mathcal{M}), {\underrightarrow{\lim}}\mathcal{M})=0$ for every $k \in \lbrace 1,\ldots,n-1 \rbrace$.
\item[$(d)$]  $\lbrace X\in  \Lambda$-{\rm Mod}$~|~ {\rm Ext}^i_{\Lambda}(\mathcal{M},X)=0 ~~~ {\rm for ~all}~ 0<i<n \rbrace \subseteq {\underrightarrow{\lim}}\mathcal{M}$.
\end{itemize}
\end{Pro}
\begin{proof}
$(a)$. We already know that every module is a factor module of a free module and every module is an essential submodule of an injective module. From $\Lambda , D(\Lambda)  \in \mathcal{M}$, for each left $\Lambda$-module $M$, we obtain homomorphisms $P \rightarrow M$ and $M \rightarrow E$ which are epimorphism and monomorphism, respectively, where $P, E \in {\rm Add}(\mathcal{M})$. Hence ${\underrightarrow{\lim}}\mathcal{M}$ is a generating-cogenerating subcategory of $\Lambda$-{\rm Mod}.\\
$(b)$. Since $\mathcal{M}$ is a covariantly finite subcategory of $\Lambda$-mod, by \cite[Proposition 3.11]{sp}, ${\underrightarrow{\lim}}\mathcal{M}$ is a covariantly finite subcategory of $\Lambda$-{\rm Mod}. Also, by applying the technique used in the proof of $((b) \Rightarrow (c))$ of \cite[Proposition 13.3]{wi}, we have an epimorphism  $\varphi: \bigoplus_{Q \in {\underrightarrow{\lim}}\mathcal{M}}Q \rightarrow M$ with the property that each homomorphism from a module $Q' \in {\underrightarrow{\lim}}\mathcal{M}$ to $M$ factorizes through $\varphi$. Since by \cite[Theorem 4.1]{wc}, ${\underrightarrow{\lim}}\mathcal{M}$ is an additive category with direct limit,  $\varphi$ is a right ${\underrightarrow{\lim}}\mathcal{M}$-approximation for $M$. This implies that ${\underrightarrow{\lim}}\mathcal{M}$ is a contravariantly finite subcategory of $\Lambda$-{\rm Mod} and the result follows.\\
$(c)$. By using \cite[Proposition 7.21]{r} and \cite[Proposition 3.1.16]{e}, for any $N \in {\rm Add}(\mathcal{M})$ and $M \in {\underrightarrow{\lim}}\mathcal{M}$, we have ${\rm Ext}_{\Lambda}^k(N,M) \cong \prod_{i \in A}{\underrightarrow{\lim}}_{j \in J} {\rm Ext}_{\Lambda}^k(N_i,M_j)$, where each $N_i, M_j \in \mathcal{M}$. Since $\mathcal{M}$ is n-rigid, ${\rm Ext}_{\Lambda}^{k}(N, M)=0$ for every $k \in \lbrace 1,\cdots,n-1 \rbrace$ and the result follows.\\
$(d)$. Let $\lbrace U_{\alpha}~|~ \alpha \in J \rbrace$ be a complete set of representative of the isomorphic classes of indecomposable modules in $\mathcal{M}$ and $T=\widehat{{\rm End}}_{\Lambda}(U)$, where $U=\bigoplus _{\alpha \in J} U_{\alpha}$. Assume that $X$ is a left $\Lambda$-module such that ${\rm Ext}_{\Lambda}^{k}(\mathcal{M}, X)=0$ for all $k \in \lbrace 1, \cdots,n-1 \rbrace$. We show that $X \in {\underrightarrow{\lim}}\mathcal{M}$. Consider the minimal injective resolution of $X$, namely
$$0 \rightarrow X \rightarrow E_0 \rightarrow E_1 \rightarrow \cdots \rightarrow E_n.$$
Applying the functor $\widehat{{\Hom}}_{\Lambda}(U,-)$ and using the fact that ${\rm Ext}_{\Lambda}^{k}(\mathcal{M}, X)=0$ for all ${k \in \lbrace 1, \cdots, n-1 \rbrace}$, we obtain an exact sequence 
\begin{center}
$0 \rightarrow \widehat{{\Hom}}_{\Lambda}(U,X) \rightarrow \widehat{{\Hom}}_{\Lambda}(U,E_0) \rightarrow \widehat{{\Hom}}_{\Lambda}(U,E_1) \rightarrow \cdots \rightarrow \widehat{{\Hom}}_{\Lambda}(U,E_n).$
\end{center} 
Since each $E_l \in {\rm Add}(\mathcal{M})$, each $\widehat{{\Hom}}_{\Lambda}(U,E_l) \in {\rm Proj}(T)$ and so by Theorem \ref{1.2}, we can see that $\widehat{{\Hom}}_{\Lambda}(U,X) \in {\rm Flat}(T)$. Hence $U \otimes_T \widehat{{\Hom}}_{\Lambda}(U,X) \in {\underrightarrow{\lim}}\mathcal{M}$. On the other hand, by using the fact that $\Lambda \in \mathcal{M}$, it is not difficult to see that $U \otimes_T \widehat{{\Hom}}_{\Lambda}(U,X) \cong X$. Therefore $X \in {\underrightarrow{\lim}}\mathcal{M}$ and the result follows.  
\end{proof}

As a consequence, we have the following result.

\begin{Cor}\label{2.1.1}
Let $\mathcal{M}$ be an $n$-cluster tilting subcategory of $\Lambda$-mod with $n \geq 2$. Then $\lbrace X\in  \Lambda$-{\rm Mod}$~|~ {\rm Ext}^i_{\Lambda}(\mathcal{M},X)=0 ~~~ {\rm for ~all}~ 0<i<n \rbrace = {\underrightarrow{\lim}}\mathcal{M}$.
\end{Cor}

Let $\mathcal{M}$ be an $n$-cluster tilting subcategory of $\Lambda$-mod with $n \geq 2$. Then by the proof of Proposition \ref{2.1} (a) and (b),  ${\rm Add}(\mathcal{M})$ is a generating-cogenerating contravariantly finite subcategory of $\Lambda$-Mod. Also, by Proposition \ref{2.1} (c), ${\rm Add}(\mathcal{M})$ is an $n$-rigid subcategory of $\Lambda$-Mod. Therefore it is natural to ask when ${\rm Add}(\mathcal{M})$ is a maximal $n$-rigid subcategory of $\Lambda$-Mod. Hence we prove the following theorem. 

\begin{The}\label{3.1}
Let $\mathcal{M}$ be an $n$-cluster tilting subcategory of $\Lambda$-mod with $n \geq 2$. Then the following statements are equivalent.
\begin{itemize}
\item[$(a)$] $\mathcal{M}$ is of finite type.
\item[$(b)$] $\lbrace X\in  \Lambda$-{\rm Mod}$~|~ {\rm Ext}^i_{\Lambda}(\mathcal{M},X)=0 ~~~ {\rm for ~all}~ 0<i<n \rbrace = {\rm Add}(\mathcal{M})$.
\item[$(c)$] $\lbrace X\in  \Lambda$-{\rm Mod}$~|~ {\rm Ext}^i_{\Lambda}(\mathcal{M},X)=0 ~~~ {\rm for ~all}~ 0<i<n \rbrace \subseteq {\rm Add}(\mathcal{M})$.
\item[$(d)$] ${\rm Add}(\mathcal{M})$ is an $n$-cluster tilting subcategory of $\Lambda$-{\rm Mod}.
\end{itemize}
\end{The}
\begin{proof}
$(a) \Rightarrow (b)$.  Since $\mathcal{M}$ is of finite type, $\mathcal{M}$ has an additive generator. Let $\lbrace U_1, \cdots, U_m \rbrace$ be a complete set of non-isomorphic indecomposable modules in $\mathcal{M}$ and $T={\rm End}_{\Lambda}(U)$, where $U= \bigoplus_{i=1}^mU_i$. Hence $T$ is an artin algebra and so ${\rm Proj}(T)={\rm Flat}(T)$.  Let $X \in {\underrightarrow{\lim}}\mathcal{M}$. Then $X ={\underrightarrow{\lim}}M_i$, where each $M_i \in \mathcal{M}$. Since $U$ is finitely generated, ${\rm Hom}_{\Lambda}(U, X) \cong {\underrightarrow{\lim}} {\rm Hom}_{\Lambda}(U, M_i)$ as $T$-modules. But because ${\rm Hom}_{\Lambda}(U, -)$ induces an additive equivalence between the full subcategory $\mathcal{M}$ of $\Lambda$-mod and and the full subcategory ${\rm proj}(T)$ of $T$-mod with the inverse equivalence $U \otimes_T -$, we have ${\rm Hom}_{\Lambda}(U, X) \in {\rm Flat}(T)$ and so it is a projective left $T$-module. Since ${\rm Hom}_{\Lambda}(U, -): {\rm Add}(\mathcal{M}) \rightarrow {\rm Proj}(T)$ is an equivalence with the inverse equivalence $U \otimes_T -$, $U \otimes_T {\rm Hom}_{\Lambda}(U, X) \in {\rm Add}(\mathcal{M})$. But since $U \otimes_T {\rm Hom}_{\Lambda}(U, X) \cong X$, so $X \in {\rm Add}(\mathcal{M})$. Therefore ${\rm Add}(\mathcal{M})={\underrightarrow{\lim}}\mathcal{M}$. Hence by Corollary \ref{2.1.1}, $\lbrace X\in  \Lambda$-{\rm Mod}$~|~ {\rm Ext}^i_{\Lambda}(\mathcal{M},X)=0 ~~~ {\rm for ~all}~ 0<i<n \rbrace = {\rm Add}(\mathcal{M})$.\\
$(b) \Rightarrow (c)$ is clear.\\
$(c) \Rightarrow (d)$. By Corollary \ref{2.1.1} and assumption, ${\underrightarrow{\lim}}\mathcal{M} = {\rm Add}(\mathcal{M})$ and  ${{\rm Add}(\mathcal{M})}^{~{\bot}_{n}} \subseteq {\rm Add}(\mathcal{M})$. Hence by Proposition \ref{2.1} and \cite[Proposition 2.2]{I3}, the result follows.\\
$(d) \Rightarrow (a)$ follows from \cite[Theorem 4.12]{en}.
\end{proof}

\begin{Rem}{\rm 
In \cite{I3}, Iyama has been posed the following question that remains unanswered:
\begin{itemize}
\item[$\bullet$] Let $\Lambda$ be a finite dimension $K$-algebra. Does any $n$-cluster tilting subcategory of $\Lambda$-mod with $n \geq 2$ have an additive generator?
\end{itemize}
Ebrahimi and the second author in \cite{en} posed the following question, which is equivalent to the Iyama’s question. 
\begin{itemize}
\item[$\bullet$] Let $\mathcal{M}$ be an $n$-cluster tilting subcategory of $\Lambda$-mod with $n \geq 2$. Is ${\rm Add}(\mathcal{M})$  an $n$-cluster tilting subcategory of $\Lambda$-Mod?
\end{itemize}
By Theorem \ref{3.1}, this question can be restated as:
\begin{itemize}
\item[$\bullet$] Let $\mathcal{M}$ be an $n$-cluster tilting subcategory of $\Lambda$-mod with $n \geq 2$. Does every left $\Lambda$-module $X$ with ${\rm Ext}^{1,\cdots,n-1}_{\Lambda}(\mathcal{M},X)=0$ belong to ${\rm Add}(\mathcal{M})$?
\end{itemize}
}
\end{Rem}

Let $\lbrace V_{\alpha}~|~ \alpha \in A \rbrace$ be a family of left $\Lambda$-modules and $V=\bigoplus_{\alpha \in A}V_{\alpha}$. We recall from \cite[p. 497]{wi} that a left $\Lambda$-module $M$ is called {\it $V$-supported}, if $M = {\underrightarrow{\lim}} M_{\lambda}$ with each $M_{\lambda}$ is a direct sums of $V_{\alpha}$'s.\\

Let $\mathcal{M}$ be an additive full subcategory of $\Lambda$-mod which is closed under direct summands. A sequence $0 \rightarrow A \overset{f}{\rightarrow} B \overset{g}{\rightarrow} C \rightarrow 0$ in ${\underrightarrow{\lim}}\mathcal{M}$ (i.e., a pair of maps $f$ and $g$ with $fg=0$) is said to be {\it pure-exact} if the sequence $$0 \rightarrow {\rm Hom}_{\Lambda}(M, A) \overset{{\rm Hom}_{\Lambda}(M, f)}{\rightarrow} {\rm Hom}_{\Lambda}(M, B) \overset{{\rm Hom}_{\Lambda}(M, g)}{\rightarrow} {\rm Hom}_{\Lambda}(M, C) \rightarrow 0$$ is exact for each $M \in \mathcal{M}$ (see \cite[p. 1653]{wc}).\\

A homomorphism $g: M \rightarrow X$ is called {\it right minimal} if any endomorphism $h: M \rightarrow M$ such that $hg=g$ is an isomorphism. Let $\mathcal{X}$ be a full subcategory of $\Lambda$-Mod. A left $\Lambda$-module $N$ has an {\it $\mathcal{X}$-cover} if there is a right minimal right $\mathcal{X}$-approximation $X \rightarrow N$.

\begin{Rem}{\rm 
Let $\mathcal{M}$ be an $n$-cluster tilting subcategory of $\Lambda$-mod with $n \geq 2$. Motivate by Theorem \ref{3.1} (see also \cite[Proposition 4.4]{en1}), it is natural to ask whether ${\rm Add}(\mathcal{M})$ is closed under direct limits. This question is equivalent to the Iyama’s question. Let $\lbrace U_{\alpha}~|~ \alpha \in J \rbrace$ be a complete set of non-isomorphic indecomposable modules in $\mathcal{M}$ and set $U=\bigoplus_{\alpha \in J} U_{\alpha}$. Let $0 \rightarrow N' \rightarrow N \rightarrow N'' \rightarrow 0$ be a pure exact sequence, where $N \in {\rm Add}(U)$. Then by \cite[Proposition 51.9(2)]{wi}, $N'$ is a $U$-supported module and so $N' \in {\underrightarrow{\lim}}\mathcal{M}$. Hence by Proposition \ref{2.1}, ${\rm Ext}_{\Lambda}^{1}(U, N')=0$. This implies that the sequence of abelian groups 
$$0 \rightarrow {\rm Hom}_{\Lambda}(U, N') \rightarrow {\rm Hom}_{\Lambda}(U, N) \rightarrow {\rm Hom}_{\Lambda}(U, N'') \rightarrow 0$$ is exact. Therefore by using \cite[Corollary 9.6]{bp}, the Iyama’s question is equivalent to the following question.

\begin{itemize}
\item[$\bullet$] Let $\mathcal{M}$ be an $n$-cluster tilting subcategory of $\Lambda$-mod with $n \geq 2$. Does any object of ${\underrightarrow{\lim}}\mathcal{M}$ has an ${\rm Add}(\mathcal{M})$-cover? 
\end{itemize}
}
\end{Rem} 

Let $\mathcal{M}$ be an additive full subcategory of $\Lambda$-mod which is closed under direct summands.  Following \cite{Sim6, Sim7}, the category ${\underrightarrow{\lim}}\mathcal{M}$ is defined to be {\it pure semisimple} if any pure-exact sequence in ${\underrightarrow{\lim}}\mathcal{M}$ splits (see also \cite[Section 3]{wc}).

\begin{Rem}\label{2.4.1}
{\rm Let $\Delta$ be a left artinian ring and $\mathcal{M}$ be a covariantly finite additive full subcategory of $\Delta$-mod which is closed under direct summands. Then by \cite[Theorem 4.2]{wc}, $\mathcal{A}={\underrightarrow{\lim}}\mathcal{M}$ is a locally finitely presented category with products. Hence by using \cite[Proposition 3.11 and Corollary 2.7]{sp} and \cite[Lemma 4.3]{wc}, we can see that $\mathcal{A}$ is a  pure semisimple category if and only if every module in $\mathcal{A}$ is a direct sum of indecomposable modules. Equivalently, every module in $\mathcal{A}$ is a direct sum of finitely generated modules (see \cite[Subsection 2.1.1]{ab}).}
\end{Rem}

In the following theorem, we give equivalence conditions for $n$-rigidity of filtered colimits of $n$-cluster tilting subcategories of module categories. 

\begin{The}\label{2.2}
Let $\mathcal{M}$ be an $n$-cluster tilting subcategory of $\Lambda$-mod with $n \geq 2$. Then the following statements are equivalent.
\begin{itemize}
\item[$(a)$] $\mathcal{M}$ is of finite type.
\item[$(b)$] ${\underrightarrow{\lim}}\mathcal{M}$ is an $n$-cluster tilting subcategory of $\Lambda$-{\rm Mod}.
\item[$(c)$] ${\underrightarrow{\lim}}\mathcal{M}$ is an $n$-rigid subcategory of $\Lambda$-{\rm Mod}.
\item[$(d)$] ${\rm Ext}_{\Lambda}^1({\underrightarrow{\lim}}\mathcal{M}, {\underrightarrow{\lim}}\mathcal{M})=0$.
\item[$(e)$] Every pure submodule $N$ of a module $X$ in ${\rm Add}(\mathcal{M})$ is a direct summand of $X$.
\item[$(f)$] ${\underrightarrow{\lim}}\mathcal{M}$ is pure semisimple.
\end{itemize}
\end{The}
\begin{proof}
$(a) \Rightarrow (b)$. Since $\mathcal{M}$ is of finite type, by Corollary \ref{2.1.1} and Theorem \ref{3.1},  ${\rm Add}(\mathcal{M})={\underrightarrow{\lim}}\mathcal{M}$ and ${\rm Add}(\mathcal{M})$ is an $n$-cluster tilting subcategory of $\Lambda$-{\rm Mod}. Hence ${\underrightarrow{\lim}}\mathcal{M}$ is an $n$-cluster tilting subcategory of $\Lambda$-{\rm Mod}.\\     
$(b) \Rightarrow (c)$ and  $(c) \Rightarrow (d)$ are clear.\\
$(d) \Rightarrow (e)$. Let $X$ be a module in ${\rm Add}(\mathcal{M})$ and $N$ be a pure submodule of $X$. We show that $N$ is a direct summand of $X$. Consider the pure exact sequence $$0 \rightarrow N \rightarrow X \rightarrow X/N \rightarrow 0.$$ 
Let $\lbrace U_{\alpha}~|~ \alpha \in J \rbrace$ be a complete set of representative of the isomorphic classes of indecomposable modules in $\mathcal{M}$ and $U=\bigoplus _{\alpha \in J} U_{\alpha}$. Since $X$ is $U$-supported, by \cite[Proposition 51.9(2)]{wi}, $N$ and $X/N$ are $U$-supported modules. Since by \cite[Theorem 4.1]{wc}, ${\underrightarrow{\lim}}\mathcal{M}$ is an additive category with direct limit, $N$ and $X/N$ are modules in ${\underrightarrow{\lim}}\mathcal{M}$. Since ${\rm Ext}_{\Lambda}^1({\underrightarrow{\lim}}\mathcal{M}, {\underrightarrow{\lim}}\mathcal{M})=0$, the exact sequence $0 \rightarrow N \rightarrow X \rightarrow X/N \rightarrow 0$ is split and so the result follows.\\
$(e) \Rightarrow (f)$. By Remark \ref{2.4.1}, it is enough to show that every module in ${\underrightarrow{\lim}}\mathcal{M}$ is a direct sum of finitely generated modules. Let $X$ be a module in ${\underrightarrow{\lim}}\mathcal{M}$. Then by \cite[Lemma 4.1]{wc}, there exists a pure epimorphism $\varphi: \bigoplus_{\alpha \in A}M_{\alpha} {\rightarrow} X$ with each $M_{\alpha} \in \mathcal{M}$. This implies that ${\rm Ker}(\varphi)$ is a pure submodule of $\bigoplus_{\alpha \in A}M_{\alpha}$. Hence by assumption, $X$ is a direct sum of finitely generated modules and so the result follows.\\
$(f) \Rightarrow (a)$ follows from \cite[Corollary 4.6]{fin}.
\end{proof}

\begin{Rem}\label{2.3}{\rm 
Let $\mathcal{M}$ be an $n$-cluster tilting subcategory of $\Lambda$-mod with $n \geq 2$. Suppose that ${(M_i, \varphi_{ij})_{i,j \in I}}$ is a direct system in $\mathcal{M}$. Consider the following property
\begin{itemize}
\item[$(\ast)$] For each family $\lbrace \psi_{ij}: M_i \rightarrow N~|~i < j \rbrace$ of morphisms with $N \in  {\underrightarrow{\lim}}\mathcal{M}$ such that 
\begin{center}
$\psi_{ik} = \psi_{ij} + \varphi_{ij}\psi_{jk}$ \hspace{1cm} for each $i, j, k \in I$ with $i < j< k$,
\end{center}
 there exists a family $\lbrace \psi_i: M_i \rightarrow N~|~i \in I \rbrace$ of morphisms such that 
 \begin{center}
 $\psi_i=\psi_{ij} + \varphi_{ij}\psi_j$ \hspace{1cm} for each $i, j \in I$ with $i < j$.
 \end{center}
\end{itemize}
By Theorem \ref{2.2} and \cite[Lemma 2.3]{sa}, the Iyama’s question equivalent to the following question:
\begin{itemize}
\item Let $\mathcal{M}$ be an $n$-cluster tilting subcategory of $\Lambda$-mod with $n \geq 2$.  Does any direct system ${(M_i, \varphi_{ij})_{i,j \in I}}$ in $\mathcal{M}$ have the property $(\ast)$?
\end{itemize}
}
\end{Rem} 

A left $\Lambda$-module $M$ is called {\it Mittag-Leffler} if the canonical map $$ (\prod_{i \in I} Q_i) \otimes_{\Lambda} M \rightarrow \prod_{i \in I} (Q_i \otimes_{\Lambda}M)$$ is injective, where $\lbrace Q_i~ |~ i \in I\rbrace$ are arbitrary right $\Lambda$-modules (see \cite[p. 19]{az}). A direct system $(M_j, \varphi_{ij})_{i, j \in I}$ of left $\Lambda$-module is called {\it factorizable} if for every $i \in I$ there exists $j \geq i$ such that $\varphi_{ij}$ factors through each $\varphi_{ik}$ with $k \geq i$. Let $(M_i, \varphi_{ij})_{i, j \in I}$ be a direct system of finitely generated left $\Lambda$-modules and $M={\underrightarrow{\lim}} M_i$. Then by \cite[Propositions 2.1.4 and 2.1.5]{rg}, $M$ is Mittag-Leffler if and only if $(M_j, \varphi_{ij})_{i, j \in I}$ is factorizable.  For artin algebras, it is known that Mittag–Leffler modules coincide with locally pure projective modules (see \cite[p. 3]{aht} and also \cite[Theorem 5]{az}). Recall that a left $\Lambda$-module $N$ is called {\it locally pure projective} if any pure epimorphism $g: Y \rightarrow N$ is locally split, that is, for each $x \in N$ there is a homomorphism $\varphi=\varphi_{x}: N \rightarrow Y$ such that $x = (x)\varphi g$.\\  

Let $\mathcal{M}$ be a skeletally small additive full subcategory of $\Lambda$-mod which is closed under direct summands. Let $\lbrace M_i~|~i \in I\rbrace$ be a complete set of isomorphism classes of objects in $\mathcal{M}$ and put $M= \bigoplus_{i \in I}M_i$. We denote by $\mathscr{G}(M)$ the category of all left $\Lambda$-modules $N$ which admit a locally split epimorphism $g: M^{(K)} \rightarrow N$ for some set $K$. By using \cite[Lemma 2.1]{anh}, we can see that ${\rm Add}(\mathcal{M}) \subseteq \mathscr{G}(M)$. Moreover $\mathscr{G}(M)$ consists of all modules in ${\underrightarrow{\lim}}\mathcal{M}$ which are locally pure projective (see \cite[Proposition 2.3]{anh}).

\begin{Rem}\label{3.3}{\rm 
Let $\mathcal{M}$ be an $n$-cluster tilting subcategory of $\Lambda$-mod with $n \geq 2$ and $\lbrace U_{\alpha}~|~\alpha \in J\rbrace$ be a complete set of isomorphism classes of objects in $\mathcal{M}$. Therefore ${\rm Add}(\mathcal{M}) \subseteq \mathscr{G}(U) \subseteq {\underrightarrow{\lim}}\mathcal{M}$, where $U= \bigoplus_{\alpha \in J}U_{\alpha}$. The locally finitely presented category ${\underrightarrow{\lim}}\mathcal{M}$ is a covariantly finite subcategory of $\Lambda$-Mod and also it is closed under direct product (see \cite[Proposition 3.11]{sp} and \cite[Theorem 4.2]{wc}). By \cite[Corollary 3.5]{ra}, ${\rm Add}(\mathcal{M})$ is a covariantly finite subcategory of $\Lambda$-Mod if and only if ${\rm Add}(\mathcal{M})$ is closed under direct products. Note that Rothmaler in \cite{p} studied direct products of Mittag-Leffler modules by using tools in model theory. He showed that all reduced products of Mittag-Leffler left $\Lambda$-modules are Mittag-Leffler if and only if all left $\Lambda$-modules are Mittag-Leffler if and only if $\Lambda$  is left pure-semisimple (see \cite[Corollary 3.2]{p}). Hence one may ask when $\mathscr{G}(U)$ is a covariantly finite subcategory of $\Lambda$-Mod. Also, it is natural to ask whether $\mathscr{G}(U)$ is closed under direct products. In particular, it is natural and interesting to ask the following question:
\begin{itemize}
\item[$\bullet$] Let $\mathcal{M}$ be an $n$-cluster tilting subcategory of $\Lambda$-mod with $n \geq 2$. Is any direct system in $\mathcal{M}$ factorizable?
\end{itemize}
}
\end{Rem}

As a consequence of Theorem \ref{2.2} and \cite[Theorems 3.4, 4.4, 5.1 and  5.2]{anh}, we have the following result that shows the questions in Remark \ref{3.3} are equivalent to the Iyama’s question.

\begin{Cor}\label{2.3}
Let $\mathcal{M}$ be an $n$-cluster tilting subcategory of $\Lambda$-mod with $n \geq 2$ and $\lbrace U_{\alpha}~|~\alpha \in J\rbrace$ be a complete set of isomorphism classes of objects of $\mathcal{M}$. Put $U= \bigoplus_{\alpha \in J}U_{\alpha}$. Then the following statements are equivalent.
\begin{itemize}
\item[$(a)$] $\mathcal{M}$ is of finite type.
\item[$(b)$] All modules in ${\underrightarrow{\lim}}\mathcal{M}$ are Mittag-Leffler modules.
\item[$(c)$] $\mathscr{G}(U)$ is closed under direct products.
\item[$(d)$] $\mathscr{G}(U)$ is closed under direct limits.
\item[$(e)$] $\mathscr{G}(U)$ is a covariantly finite subcategory of $\Lambda$-Mod.
\item[$(f)$] Every pure submodule $N$ of a module $X \in {\rm Add}(\mathcal{M})$ is a direct summand of $X$.
\end{itemize}
\end{Cor}

\section*{acknowledgements}
The work of the authors is based upon research funded by Iran National Science Foundation (INSF) under project No. 4001480. Also, the research of the second author was in part supported by a grant from IPM (No. 1403160416). The research of the first author was in part carried out in IPM-Isfahan Branch.


\begin{thebibliography}{10}
\bibitem{ai} T. Aihara, T. Araya, O. Iyama, R.  Takahashi, and M. Yoshiwaki, {\it Dimensions of triangulated categories with respect to subcategories}, J. Algebra 399 (2014), 205--219.
\bibitem{anh} L. Angeleri-Hugel, {\it Covers and envelopes via endoproperties of modules,} Proc. Lond. Math. Soc., III. Ser. 86 (2003), no. 3, 649--665.
\bibitem{aht} L. Angeleri-Hugel, D. Herbera, and J. Trlifaj, {\it Baer and Mittag-Leffler modules over tame hereditary algebras,} Math. Z. 265 (2010), no. 1, 1--19. 
\bibitem{Ausla2} M. Auslander, {\it Representation theory of artin algebras II}, Commun. Algebra 1 (1974), 269--310.
\bibitem{ar} M. Auslander and I. Reiten, {\it Applications of contravariantly finite subcategories}, Adv. Math. 86 (1991), no. 1, 111--152.
\bibitem{as} M. Auslander and S. O. Smal$\varnothing$, {\it Almost split sequences in subcategories}, J. Algebra 69 (1981), 426--454.
\bibitem{az} G. Azumaya, {\it Locally pure-projective modules,} Contemp. Math. 124 (1992), 17-22.
\bibitem{af} G. Azumaya and A. Facchini, {\it Rings of pure global dimension zero and Mittag-Leffler modules,} J. Pure Appl. Algebra 62 (1989), no. 2, 109--122. 
\bibitem{bp} S. Bazzoni and L. Positselski, {\it Covers and direct limits: a contramodule-based approach,} Math. Z. 299 (2021), no. 1-2, 1--52.
\bibitem{ab} A. Beligiannis, {\it On algebras of finite Cohen-Macaulay type,} Adv. Math. 226 (2011), no. 2, 1973--2019.
\bibitem{b} A. Beligiannis, {\it Relative homology, higher cluster-tilting theory and categorified Auslander-Iyama
correspondence,} J. Algebra, 444 (2015), 367--503.
\bibitem{wc}  W. Crawley-Boevey, {\it Locally finitely presented additive categories,} Commun. Algebra 22 (1994), no. 5, 1641--1674.
\bibitem{de} A. Del Rio, {\it Weak dimension of group-graded rings,} Publ. Mat., Barc. 34 (1990), no. 1, 209--216.
\bibitem{dn}  R. Diyanatnezhad and A. Nasr-Isfahani, {\it Relations for Grothendieck groups of $n$-cluster tilting subcategories,}
J. Algebra, 594 (2022), 54--73.
\bibitem{dn1} R. Diyanatnezhad and A. Nasr-Isfahani, {\it The radical of functorially finite subcategories,} J. Algebra, 657 (2024), 675--703.
\bibitem{en} R. Ebrahimi and A. Nasr-Isfahani, {\it Pure semisimple n-cluster tilting subcategories,} J. Algebra, 549 (2020), 177--194.
\bibitem{en1} R. Ebrahimi and A. Nasr-Isfahani, {\it The completion of d-abelian categories,} J. Algebra 645 (2024), 143--163. 
\bibitem{e} E. E. Enochs, O. M. G. Jenda, {\it Relative Homological Algebra,} de Gruyter Exp. Math., vol. 30, Walter de Gruyter and Co., Berlin, 2000, xii+339 pp.
\bibitem{fin} Z. Fazelpour and A. Nasr-Isfahani, {\it Finiteness and purity of subcategories of the module categories}, arXiv:2203.03294.
\bibitem{zn}  Z. Fazelpour and A. Nasr-Isfahani, {\it Morita equivalence and Morita duality for rings with local units and the subcategory of projective unitary modules}, Appl. Categ. Struct. 32 (2024), no. 2, Paper no. 10, 28 p. 
\bibitem{fu} K. R. Fuller, {\it On rings whose left modules are direct sums of finitely generated modules}, Proc. Am. Math. Soc. 54 (1976), 39--44.
\bibitem{f} K. R. Fuller and H. Hullinger, {\it Rings with finiteness conditions and their categories of functors}, J. Algebra 55 (1978), 94--105.
\bibitem{gs} J. L. Garcia and D. Simson, {\it On rings whose flat modules form a Grothendieck category,} Colloq.
Math. 73 (1997), no. 1, 115-141.
\bibitem{gro} A. Grothendieck, {\it Éléments de géométrie algébrique. III: Étude cohomologique des faisceaux cohérents. (Séconde partie). Rédigé avec la colloboration de J. Dieudonné,} Publ. Math., Inst. Hautes Étud. Sci. 17 (1963), 137--223.
\bibitem{rg} L. Gruson and M. Raynaud, {\it Critères de platitude et de projectivité. Techniques de 'platification' d'un module,} Invent. Math. 13 (1971), 1-89.
\bibitem{I3} O. Iyama, {\it Auslander-Reiten theory revisited,} In Trends in Representation Theory
of Algebras and Related Topics, {} (2008), 349--397.
\bibitem{I2} O. Iyama, {\it Higher-dimensional Auslander-Reiten theory on maximal orthogonal subcategories,} Adv. Math., 210 (2007), no. 1, 22--50.
\bibitem{J} G. Jasso, {\it n-Abelian and n-exact categories,} Math. Z, 283 (2016), no. 3-4 , 703-759.
\bibitem{sp} H. Krause, {\it The spectrum of a module category},  Mem. Am. Math. Soc. 707 (2001), x+125 pp.
\bibitem{men} C. Menini, {\it Gabriel-Popescu Type Theorems and Graded Modules,} Perspectives in ring theory, Proc. NATO Adv. Res. Workshop, Antwerp/Belg. 1987, NATO ASI Ser., Ser. C 233 (1988), 239-251.
\bibitem{mp} M. Prest, {\it Purity, Spectra and Localisation,} Encyclopedia of Mathematics and its Applications, vol. 121, Cambridge University Press, Cambridge, 2009.
\bibitem{ra} J. Rada and M. Saorin, {\it Rings characterized by (pre)envelopes and (pre)covers of their modules,} Commun. Algebra 26 (1998), no. 3, 899--912.
\bibitem{p} P. Rothmaler, {\it Mittag-Leffler modules,} Ann. Pure Appl. Logic 88 (1997), no. 2-3, 227--239.
\bibitem{r} J. Rotman, {\it An Introduction to Homological Algebra,} Springer, New York, second edition, 2009.
\bibitem{sa} J. Saroch and J. Stovicek, {\it Singular compactness and definability for $\sum$-cotorsion and Gorenstein modules,} Sel. Math., New Ser. 26 (2020), no. 2, Paper no. 23, 40 p.
\bibitem{s77} D. Simson, {\it On pure global dimension of locally finitely presented Grothendieck categories}, Fundam. Math. 96 (1977), 91--116.
\bibitem{Sim6} D. Simson, {\it Pure semisimple categories and rings of finite representation type}, J. Algebra 48 (1977), 290--295.
\bibitem{Sim7} D. Simson, {\it Pure semisimple categories and rings of finite representation type, Corrigendum}, J. Algebra 67(1) (1980), 254--256.
\bibitem{wi} R. Wisbauer,  {\it Foundations of module and ring theory},  A handbook for study and research. Revised and translated from the 1988 German edition. Algebra, Logic and Applications, 3. Gordon and 3. Gordon and Breach Science Publishers, Philadelphia, PA, 1991.
\bibitem{y} K. Yamagata, {\it On Morita duality for additive group valued functors}, Commun. Algebra 7 (1979), 367-392.
\end{thebibliography}
\end{document}